\documentclass[seceqn,secthm]{elsart}

\usepackage[all]{xy}
\usepackage{pifont}
\usepackage{amssymb,amsmath,amsfonts,amsbsy}
\usepackage[numbers]{natbib}
\newenvironment{proof}{\paragraph*{\it Proof.}}{\halmos}

\renewcommand{\leq}{\leqslant}
\renewcommand{\geq}{\geqslant}

\newcommand{\eps}{\varepsilon}

\newcommand{\inv}{^\leftarrow}
\newcommand{\RR}{\mathbb{R}}
\newcommand{\halmos}{\mbox{} \hfill $\Box$}

\newcommand{\sbigcap}{{\textstyle\bigcap}}
\newcommand{\sbigcup}{{\textstyle\bigcup}}
\newcommand{\ssum}{{\textstyle\sum}}
\newcommand{\dsty}{\displaystyle}

\newcommand{\ub}{\boldsymbol{u}}
\newcommand{\vb}{\boldsymbol{v}}

\newcommand{\xb}{\boldsymbol{x}}
\newcommand{\yb}{\boldsymbol{y}}

\newcommand{\rmd}{\mathrm{d}}
\newcommand{\dt}{\rmd t}
\newcommand{\du}{\rmd u}
\newcommand{\dv}{\rmd v}

\begin{document}

\begin{frontmatter}

\title{Tails of Multivariate Archimedean Copulas}

\author[France]{Arthur Charpentier}
\ead{arthur.charpentier@univ-rennes1.fr}
\address[France]{CREM - Universit\'e Rennes 1, Facult\'e des Sciences Economiques \\ 
7, place Hoche, 35065 Rennes cedex, France}

\author[LLN]{Johan Segers\thanksref{VENI}}
\ead{johan.segers@uclouvain.be}
\thanks[VENI]{Corresponding author. Supported by a VENI grant of the Netherlands Organization for Scientific Research (NWO) and by the IAP research network grant no. P6/03 of the Belgian government (Belgian Science Policy).}
\address[LLN]{Universit\'e catholique de Louvain, Institut de statistique \\
Voie du Roman Pays 20, B-1348 Louvain-la-Neuve, Belgium}

\begin{abstract}
A complete and user-friendly directory of tails of Archimedean copulas is presented which can be used in the selection and construction of appropriate models with desired properties. The results are synthesized in the form of a decision tree: Given the values of some readily computable characteristics of the Archimedean generator, the upper and lower tails of the copula are classified into one of three classes each, one corresponding to asymptotic dependence and the other two to asymptotic independence. For a long list of single-parameter families, the relevant tail quantities are computed so that the corresponding classes in the decision tree can easily be determined. In addition, new models with tailor-made upper and lower tails can be constructed via a number of transformation methods. The frequently occurring category of asymptotic independence turns out to conceal a surprisingly rich variety of tail dependence structures.
\end{abstract}

\begin{keyword}
Archimedean copula \sep 
asymptotic independence \sep
Clayton copula \sep
coefficient of tail dependence \sep
complete monotonicity \sep
domain of attraction \sep
extreme value distribution \sep 
frailty model \sep
regular variation \sep 
survival copula \sep
tail dependence copula
\end{keyword}
\end{frontmatter}

\section{Introduction}
\label{S:intro}

A $d$-variate copula $C$ is (the restriction to $[0, 1]^d$) of a $d$-variate distribution function with uniform$(0,1)$ margins. By Sklar's theorem, a copula is what remains of an arbitrary (continuous) $d$-variate distribution function $F$ when stripped of its margins. Margin-free measures of dependence such as Kendall's~$\tau$ or Spearman's~$\rho$ depend on $F$ only through $C$. More generally, copulas form a natural way to describe dependence between variables when making abstraction of their marginal distributions. Overviews of the probabilistic and statistical aspects of copulas are to be found for instance in \citep{GF07, Joe97, McNN07, Muller-Scarsini, Nelsen99, schweizer-sklar}.

A copula $C$ is called Archimedean \citep{Genest-McKay, K74} if it is of the form
\begin{equation}
\label{E:C}
	C(u_1, \ldots, u_d) = \phi\inv(\phi(u_1) + \cdots + \phi(u_d))
\end{equation}
for $(u_1, \ldots, u_d) \in [0, 1]^d$, where the Archimedean generator $\phi : [0, 1] \to [0, \infty]$ is convex, decreasing and satisfies $\phi(1) = 0$ and where 
\[
	\phi\inv(y) = \inf \{ t \in [0, 1] : \phi(t) \leq y \}
\]
for $y \in [0, \infty]$ is the generalized inverse of $\phi$. See Table~\ref{T} for a list of some common parametric families of Archimedean generators. A necessary and sufficient condition for the right-hand side of \eqref{E:C} to specify a copula is that the function $\phi\inv$ is $d$-monotone on $(0, \infty)$, that is, $\phi\inv$ is $d-2$ times continuously differentiable, $(-D)^k \phi\inv \geq 0$ for all $k \in \{0, 1, \ldots, d-2\}$ (with $D$ the derivative operator), and $(-D)^{d-2}\phi\inv$ is convex \citep{McNN07}. Note that if $d = 2$, then the conditions on $\phi$ mentioned in the beginning of this paragraph are necessary and sufficient. If $\phi\inv$ is completely monotone, that is, if $(-D)^k \phi\inv \geq 0$ for all integer $k \geq 0$, then $\phi\inv$ is also $d$-monotone for every integer $d \geq 2$ and the resulting model can be interpreted as a frailty model \citep{MarshallOlkin88, Oakes89}. Because of their analytic tractability, Archimedean copulas have enjoyed a great popularity in applied work, from insurance \citep{Denuit, FV98} and finance \citep{CLV04, ELMcN01} to hydrology \citep{GF07, GS06, Zhang-Singh} and survival analysis \citep{Bassan-Spizzichino, DBL97, Wang03}, to mention just a few references. More flexible models can be obtained by forming mixtures of Archimedean copulas \citep{Patton}, or by constructing hierarchical (or nested) Archimedean copulas \citep{McNeil}; however, this will not be considered in this paper.

Our interest in this paper is in the tail behaviour of Archimedean copulas, that is, in the asymptotic behaviour of 
\begin{eqnarray*}
	C(u_1, \ldots, u_d) &=& \Pr[U_1 \leq u_1, \ldots, U_d \leq u_d], \\
	\overline{C}(u_1, \ldots, u_d) &=& \Pr[U_1 \geq 1 - u_1, \ldots, U_d \geq 1 - u_d]
\end{eqnarray*}
when some of the $u_i$ tend to zero. Here, $(U_1, \ldots, U_d)$ is a random vector with distribution function $C$ and survival copula $\overline{C}$. Note that $\overline{C}$ is a copula as well but is in general not an Archimedean one. We will refer to the lower and the upper tail of $C$, respectively. Knowledge of these tails is of obvious interest when modeling the joint occurrence of extremes of components of a random vector $(X_1, \ldots, X_d)$ with distribution function $F$ having copula $C$.

Aim of this paper is to give a complete taxonomy of the possible categories of upper and lower tail dependence of multivariate Archimedean copulas. For each of the two tails, there are two categories, called asymptotic dependence and asymptotic independence. In case of asymptotic dependence, the tail behavior is determined by a single parameter, the index of regular variation of the generator $\phi$ near $0$ (lower tail) or $1$ (upper tail); see \citep{ballerini, CFG00}. The case of asymptotic independence, however, has not yet been explored in the literature, although all popular parametric families exhibit asymptotic independence in at least one of the two tails. The main original contribution of this paper is therefore to fill in this gap and to give a detailed coverage of the precise asymptotics in case of asymptotic independence. The label asymptotic independence turns out to conceal a surprisingly rich variety of tail dependence structures, of which we shall give a systematic exposition. For general multivariate distributions, the complexity of the category of asymptotic independence has been recognized already since the work by Ledford and Tawn \citep{Ledford96, Ledford97}; see also \citep{Maulik04, Resnick02}.

The upshot of our investigations is the decision tree in Figure~\ref{F}, to be explained in Section~\ref{S:examples}: given the values of some readily computable quantities defined in terms of $\phi$, the upper and lower tail can be classified into one of three classes each, one corresponding to asymptotic dependence and the other two to asymptotic independence. The detailed description of each of these classes constitutes the body of the paper. For all of the families in Table~\ref{T}, the relevant generator characteristics have been computed so that the corresponding classes in the decision tree can easily be determined. In addition, new models with tailor-made upper and lower tails can be constructed via the transformation methods in Table~\ref{T:transform}. In this way, we hope to provide the reader a complete and user-friendly directory of tails of Archimedean copulas which can be used in the selection and construction of appropriate models with desired properties.

Tail behaviour can be studied from various perspectives. In this paper, we will focus mainly on the asymptotic behaviour of the joint distribution and survival functions and on the asymptotic conditional distributions of $(U_1, \ldots, U_d)$ given that all $U_i$ are close to $0$ or $1$ for all $i$ in some subset $I$ of $\{1, \ldots, d$\}. From these results, other interesting tail quantities can be derived via standard methods: minimal and maximal domains of attraction \citep{ballerini, CFG00}, tail dependence copulas \citep{JW02, CS07}, coefficients of tail dependence \citep{coles-99, Ledford96, Mesfioui-Quessy}, and tails of sums \citep{Albrecher, ALW04, ALW05, Barbe, Kortschak, Wuthrich03}. For reasons of brevity, we will not mention these explicitly.

The outline of the paper is as follows. An overview of the classes of tail dependence is presented in Section~\ref{S:examples}, together with a list of examples. Detailed results for the various classes of lower and upper tails are given in Sections~\ref{S:lower} and \ref{S:upper} and with proofs in Sections~\ref{S:lower:proofs} and~\ref{S:upper:proofs}, respectively. Appendices~\ref{S:RV} and \ref{S:form} contain some useful auxiliary results.

Throughout, $(U_1, \ldots, U_d)$ is a random vector with joint distribution function $C$, an Archimedean copula with generator $\phi$. Let $\phi\inv$ denote the generalized inverse of $\phi$ and let $\phi'$ be a non-decreasing version of the Radon-Nikodym derivative of $\phi$. Minima and maxima will be denoted by $\wedge$ and $\vee$, respectively.

\section{Overview and examples}
\label{S:examples}

\begin{figure}[f]
\begin{center}
\xymatrix{
\fbox{Which tail?}
\ar `d[dr] [dr] ^(.2){\mbox{\em upper tail}} 
\ar `d[ddddr] [ddddr] ^(.2){\mbox{\em lower tail}} \\
& \fbox{Compute $\phi'(1)$}
\ar[r]^(0.6){\phi'(1) < 0} 
\ar[d]^{\phi'(1) = 0} & 
\fbox{go to \ding{192}} \\
& \fbox{Compute $\dsty \theta_1 = - \lim_{s \downarrow 0} \frac{s \phi'(1-s)}{\phi(1-s)}$}
\ar[r]^>>>>>>{\theta_1 = 1} 
\ar `d[dr] [dr] ^{\theta_1 > 1} &
\fbox{go to \ding{193}} \\
&& \fbox{go to \ding{194}} \\
& \fbox{Compute $\phi(0)$}
\ar[r]^(0.6){\phi(0) < \infty}
\ar[d]^{\phi(0) = \infty} &
\fbox{go to \ding{195}} \\
& \fbox{Compute $\dsty \theta_0 = - \lim_{s \downarrow 0} \frac{s \phi'(s)}{\phi(s)}$}
\ar[r] ^>>>>>>{\theta_0 = 0} 
\ar `d[dr] [dr] ^{\theta_0 > 0} &
\fbox{go to \ding{196}} \\
&& \fbox{go to \ding{197}} 
}
\bigskip
\begin{tabular}{cllll}
\hline\hline
\emph{case} & \emph{tail} & \emph{AD/AI} & \emph{subsection} & \emph{remark} \\
\hline
\ding{192} & upper & AI & \ref{SS:upper:NAD} & 
	\begin{minipage}[t]{0.36\textwidth} \raggedright 
		automatically $\theta_1 = 1$; \\ 
		`near independence' 
	\end{minipage} \\
\ding{193} & upper & AI & \ref{SS:upper:NI} & `near asymptotic dependence' \\
\ding{194} & upper & AD & \ref{SS:upper:AD} \\
\ding{195} & lower & AI & \ref{SS:lower:AI:nonstrict} & 
	\begin{minipage}[t]{0.36\textwidth} \raggedright
		automatically $\theta_0 = 0$ and \\
		$C(s, \ldots, s) = 0$ for small $s$ 
	\end{minipage} \\
\ding{196} & lower & AI & \ref{SS:lower:AI:strict} & 
	\begin{minipage}[t]{0.4\textwidth} \raggedright
		compute the index of regular variation, $\kappa$, of $-1/D (\log \phi\inv)$
	\end{minipage} \\
\ding{197} & lower & AD & \ref{SS:lower:AD} & \\
\hline\hline
\end{tabular}
\end{center}
\caption{Categorizing the tail behaviour of an Archimedean copula. AD = asymptotic dependence; AI = asymptotic independence. See explanation in Section~\ref{S:examples}. \label{F}}
\end{figure}

\begin{table}[f]
\caption{Values for $-\phi'(1)$, $\theta_1$, $\phi(0)$, $\theta_0$ and $\kappa$ [if $\phi(0) = \infty$ and $\theta_0 = 0$] as in Figure~\ref{F} for the generators (1)--(22) of bivariate Archimedean copulas in \citet{Nelsen99}, Table~4.1, and a new model (23). See explanation in Section~\ref{S:examples}. In (1), (5) and (17), the case $\theta = 0$ is to be interpreted as the appropriate limit. Some named families: (1) Clayton/Cook-Johnson/Oakes; (3) Ali-Mikhail-Haq; (4) Gumbel-Hougaard; (5) Frank. \label{T}} 
\begin{center}
\[
\begin{array}{r@{\quad}ll@{\quad}r@{\quad}rc@{\quad}r@{\quad}r@{\quad}r}
\hline\hline 
& & & \multicolumn{2}{c}{\mbox{\em upper tail}} && \multicolumn{3}{c}{\mbox{\em lower tail}} \\
\cline{4-5} \cline{7-9} 
& \phi(t) & \mbox{\em range $\theta$} 
	& -\phi'(1) & \theta_1 
	&& \phi(0) & \theta_0 & \kappa \\
\hline 
(1) & \frac{1}{\theta}( t^{-\theta}-1) & [-1,\infty) 
	& 1 & 1 
	&& \frac{1}{(-\theta) \vee 0} & \theta \vee 0 & \cdot \\
(2) & (1 - t)^\theta & [1, \infty) 
	& \mathbf{1}(\theta = 1) & \theta 
	&& 1 & 0 & \cdot \\
(3) & \log \frac{1-\theta (1 - t) }{t} & [-1,1) 
	& 1-\theta & 1 
	&& \infty & 0 & 0 \\
(4) & (- \log t)^\theta & [1,\infty) 
	& \mathbf{1}(\theta = 1) & \theta 
	&& \infty & 0 & 1-\frac{1}{\theta} \\
(5) & -\log \frac{\mathrm{e}^{-\theta t}-1}{\mathrm{e}^{-\theta}-1} & \RR 
	& \frac{\theta}{\mathrm{e}^\theta-1} & 1 
	&& \infty & 0 & 0 \\
(6) & -\log \{ 1 - (1 - t)^\theta \} & [1, \infty) 
	& \mathbf{1}(\theta = 1) & \theta 
	&& \infty & 0 & 0 \\
(7) & - \log \{ \theta t + (1-\theta) \} & (0, 1]  
	& \theta & 1 
	&& -\log(1-\theta) & 0 & \cdot \\
(8) & \frac{1-t}{1 + (\theta - 1) t} & [1, \infty) 
	& \frac{1}{\theta} & 1 
	&& 1 & 0 & \cdot \\
(9) & \log (1 - \theta \log t) & (0, 1] 
	& \theta & 1 
	&& \infty & 0 & -\infty \\
(10) & \log (2t^{-\theta} - 1) & (0, 1] 
	& 2 \theta & 1 
	&& \infty & 0 & 0 \\
(11) & \log (2 - t^\theta) & (0,1/2] 
	& \theta & 1 
	&& \log 2 & 0 & \cdot \\
(12) & (\frac{1}{t} - 1)^\theta & [1, \infty) 
	& \mathbf{1}(\theta = 1) & \theta 
	&& \infty & \theta & \cdot \\
(13) & (1 - \log t)^\theta - 1 & (0,\infty ) 
	& \theta & 0 
	&& \infty & 0 & 1 - \frac{1}{\theta} \\
(14) & (t^{-1/\theta} - 1)^\theta & [1,\infty) 
	& \mathbf{1}(\theta = 1) & \theta 
	&& \infty & 1 & \cdot \\
(15) & (1 - t^{1/\theta})^\theta & [1,\infty ) 
	& \mathbf{1}(\theta = 1) & \theta 
	&& 1 & 0 & \cdot \\
(16) & (\frac{\theta}{t} + 1) (1 - t) & [0,\infty) 
	& 1+\theta & 1 
	&& \infty & 1 & \cdot \\
(17) & - \log \frac{(1 + t)^{-\theta} - 1}{2^{-\theta} - 1} & \RR 
	& \frac{\theta}{2(2^\theta - 1)} & 1 
	&& \infty & 0 & 0  \\
(18) & \mathrm{e}^{\theta / (t-1)} & [2, \infty) 
	& 0 & \infty 
	&& \mathrm{e}^{-\theta} & 0 & \cdot \\
(19) & \mathrm{e}^{\theta / t} - \mathrm{e}^\theta & (0, \infty) 
	& \theta \mathrm{e}^\theta & 1 
	&& \infty & \infty & \cdot \\
(20) & \mathrm{e}^{t^{-\theta}} - \mathrm{e} & (0, \infty) 
	& \theta \mathrm{e} & 1 
	&& \infty & \infty & \cdot \\
(21) & 1 - \{1 - (1 - t)^\theta\}^{1/\theta} & [1, \infty) 
	& \mathbf{1}(\theta = 1) & \theta 
	&& 1 & 0 & \cdot \\
(22) & \arcsin (1 - t^\theta) & (0,1] 
	& \theta & 1 
	&& \pi/2 & 0 & \cdot \\
(23) & \frac{1-t}{\{-\log(1-t)\}^\theta} & (0, \infty)
	& 0 & 1 
	&& \infty & \theta & \cdot \\
\hline\hline
\end{array}
\]
\end{center}
\end{table}

\begin{table}[f]
\caption{Values for $\phi_\alpha'(1)$, $\theta_1(\alpha)$, $\phi_\alpha(0)$, $\theta_0(\alpha)$ and $\kappa(\alpha)$ (if applicable) for transformation families $\phi_\alpha$ based on a fixed Archimedean generator $\phi$ in terms of the corresponding quantities for $\phi$ itself. See explanation in Section~\ref{S:examples}. For (2), no general formulas exist for $\theta_0(\alpha)$ and $\kappa(\alpha)$. \label{T:transform}} 
\begin{center}
\[
\begin{array}{r@{\quad}ll@{\quad}r@{\quad}rc@{\quad}r@{\quad}r@{\quad}r}
\hline\hline 
& & & \multicolumn{2}{c}{\mbox{\em upper tail}} && \multicolumn{3}{c}{\mbox{\em lower tail}} \\
\cline{4-5} \cline{7-9} 
& \phi_\alpha(t) & \mbox{\em range $\alpha$} 
	& \phi_\alpha'(1) & \theta_1(\alpha) 
	&& \phi_\alpha(0) & \theta_0(\alpha) & \kappa(\alpha) \\
\hline 
(1) & (\phi(t))^\alpha & (1, \infty)
	& 0 & \alpha \theta_1 
	&& (\phi(0))^\alpha  & \alpha \theta_0 & \frac{\kappa}{\alpha} + 1 - \frac{1}{\alpha} \\
(2) & \frac{\mathrm{e}^{\alpha \phi(t)} - 1}{\alpha} & (0, \infty)
	& \alpha \phi'(1) & \theta_1 
	&& \frac{\mathrm{\e}^{\alpha \phi(0)} - 1}{\alpha} & \ast & \ast \\
(3) & \phi(t^\alpha) & (0, 1)
	& \alpha \phi'(1) & \theta_1 
	&& \phi(0) & \alpha \theta_0 & \kappa \\
(4) & \phi(1-(1-t)^\alpha) & (1, \infty)
	& 0 & \alpha \theta_1 
	&& \phi(0) & \theta_0 & \kappa \\
(5) & \phi(\alpha t) - \phi(\alpha) & (0, 1)
	& \alpha \phi'(\alpha) & 1 
	&& \phi(0) - \phi(\alpha) & \theta_0 & \kappa \\
\hline\hline
\end{array}
\]
\end{center}
\end{table}

Our taxonomy of the upper and lower tails of Archimedean copulas is summarised in the decision tree in Figure~\ref{F}. For the upper tail, there are three categories (\ding{192} to \ding{194}, Section~\ref{S:upper}), depending on the behaviour of $\phi$ near $1$. For the lower tail, there are three categories as well (\ding{195} to \ding{197}, Section~\ref{S:lower}), depending on the behaviour of $\phi$ near $0$. For each of the six cases, the number of the relevant subsection with more detailed explanation is mentioned in the table just below the decision tree.

We have applied our taxonomy to the list of 23 one-parameter models of Archimedean generators in Table~\ref{T}. Except for the last one, the models are taken from Table~4.1 in \citet{Nelsen99}. For each model, the discriminating quantities for the decision tree in Figure~\ref{F} are listed in Table~\ref{T}. For some models, the final outcome \ding{192}, \ldots, \ding{197} in the decision tree depends on the value of the parameter; therefore, these outcomes have not been mentioned in Table~\ref{T}. Case \ding{193} where $\phi'(1) = 0$ and $\theta_1 = 1$ (Subsection~\ref{SS:upper:NAD}) does not occur for the models (1)--(22) in \citet{Nelsen99}, Table~4.1. Therefore we added the model (23), which to our knowledge is new.

In \citet{GGR98}, Proposition~1, a number of recipes are given to generate families of (bivariate) Archimedean generators out of a single such generator $\phi$. Five such transformation families are listed in Table~\ref{T:transform}. Note that families (1) and (2) are of the form $\phi_\alpha = f_\alpha \circ \phi$ where $f_\alpha : [0, \infty] \to [0, \infty]$ is a convex increasing bijection, while (3) and (4) are of the form $\phi_\alpha = \phi \circ g_\alpha$, where $g_\alpha : [0, 1] \to [0, 1]$ is a concave increasing bijection. For all of these families and wherever possible, the relevant tail quantities of the transformed generator $\phi_\alpha$ have been expressed in terms of those of the base generator $\phi$. Note that many of the models in Table~\ref{T} are examples of such transformation families based on either $\phi(t) = - \log t$ (independent copula) or $\phi(t) = 1-t$ (countermonotone copula). Further, these transformations can be combined yielding multi-parameter families; for instance, in \citep{GGR98}, a three-parameter family is constructed encompassing the Clayton, Gumbel, and Frank families.

\section{Lower tail}
\label{S:lower}

Let $C$ be an Archimedean copula with generator $\phi$. Results in this section concern the behaviour of the copula $C(u_1, \ldots, u_d)$ when at least one of the coordinates $u_i$ tends to $0$. Relevant is the asymptotic behaviour of the generator $\phi$ in the neighbourhood of $0$.

We assume the existence in $[0, \infty]$ of the limit 
\begin{equation}
\label{E:lower:AD:RV}
	\theta_0 := - \lim_{s \downarrow 0} \frac{s \phi'(s)}{\phi(s)}.
\end{equation}
The limit indeed exists for virtually every known parametric model. By the monotone density theorem (Lemma~\ref{L:RV}), equation~\eqref{E:lower:AD:RV} is equivalent to regular variation of $\phi$ at $0$ with index $-\theta_0$: 
\begin{equation}
\label{E:lower:RV}
	\lim_{s \downarrow 0} \frac{\phi(st)}{\phi(s)} = t^{-\theta_0}, \qquad t \in (0, \infty).
\end{equation}
If $\theta_0 = \infty$, the limit is to interpreted as $\infty$, $1$, or $0$ according to whether $t < 1$, $t = 1$, or $t > 1$.

There are two categories: if $\theta_0 > 0$, then the lower tail exhibits \emph{asymptotic dependence} (Subsection~\ref{SS:lower:AD}), while if $\theta_0 = 0$, then there is \emph{asymptotic independence}. Note that for non-strict generators, i.e.\ $\phi(0) < \infty$, not only $\theta_0 = 0$ but there even exists $s_0 \in (0, 1]$ such that $\Pr[U_i \leq s, U_j \leq s] = 0$ for all $s \in [0, s_0]$ and $1 \leq i < j \leq d$ (Subsection~\ref{SS:lower:AI:nonstrict}). More interesting is the case where $\phi(0) = \infty$ and $\theta_0 = 0$ (Subsection~\ref{SS:lower:AI:strict}). Here, the precise behaviour of the lower tail is described by the index of regular variation at infinity, $\kappa$, of the function $-1/D(\log \phi\inv)$, with $D$ the derivative operator.

The proofs of the theorems in this section are gathered in Section~\ref{S:lower:proofs}.

\subsection{Asymptotic dependence}
\label{SS:lower:AD}

\begin{thm}
\label{T:lower:AD}
If the limit $\theta_0$ in \eqref{E:lower:AD:RV} exists in $[0, \infty]$, then for every $I \subset \{1, \ldots, d\}$ with $|I| \geq 2$ and every $(x_i)_{i \in I} \in (0,\infty)^{|I|}$,
\begin{equation}
\label{E:lower:AD}
	\lim_{s \downarrow 0}
	s^{-1} \Pr [ \forall i \in I : U_i \leq sx_i ]
	=	\left\lbrace
        \begin{array}{l@{\quad}l}
            0 & \mbox{if $\theta_0 = 0$,} \\[1ex]
            (\sum_{i \in I} x_i^{-\theta_0})^{-1/\theta_0}
            & \mbox{if $0 < \theta_0 < \infty$,} \\[1ex]
            \bigwedge_{i \in I} x_i & \mbox{if $\theta_0 = \infty$.}
        \end{array}
	    \right.
\end{equation}
\end{thm}

By \eqref{E:lower:AD}, the index of lower tail dependence of an arbitrary pair of variables is
\[
	\lambda _{L}
	= \lim_{s \downarrow 0} \Pr[U_i \leq s \mid U_j \leq s] 
	= 2^{-1/\theta_0},
\]%
where $i \neq j$ and where $2^{-1/\theta_0}$ is to be interpreted as $0$ or $1$ if $\theta_0$ is $0$ or $\infty$, respectively. Hence, if $\theta_0 = 0$, then every pair of variables is asymptotically independent in its lower tail. In that case, more precise statements on the asymptotic behaviour of $\Pr [ \forall i \in I : U_i \leq sx_i ]$ as $s \downarrow 0$ are made in Subsections~\ref{SS:lower:AI:nonstrict} and \ref{SS:lower:AI:strict}.

Second, if $\theta_0 > 0$, then the probability that all $d$ variables are small simultaneously is of the same order as the probability that a single variable is small: for instance, if $0 < \theta_0 < \infty$, then for every pair $i, j$,
\[
	\lim_{x_i \to \infty} \lim_{s \downarrow 0} \Pr[ U_i > s x_i \mid U_j \leq s]
	= \lim_{x_i \to \infty} \{ 1 - (x_i^{-\theta_0} + 1)^{-1/\theta_0} \}
	= 0.
\]
In that case, one can compute the limit distribution as $s \downarrow 0$ of the vector $(s^{-1}U_1, \ldots, s^{-1}U_d)$ conditionally on the event that $U_i \leq s x_i$ for all $i$ in some non-empty set $I$.

\begin{cor}
\label{C:lower:AD}
If \eqref{E:lower:AD:RV} holds with $0 < \theta_0 \leq \infty$, then for every $\varnothing \neq I \subset \{1, \ldots, d\}$, every $(x_i)_{i \in I} \in (0, \infty)^{|I|}$ and every $(y_1, \ldots, y_d) \in (0, \infty)^d$,
\begin{align*}
	\lefteqn{
	\lim_{s \downarrow 0} 
	\Pr[ \forall i = 1, \ldots, d : U_i \leq s y_i \mid \forall i \in I : U_i \leq s x_i ]
	} \\
    &=	\left\lbrace
	\begin{array}{l@{\quad}l}
	\displaystyle
	\biggl(
    \frac{\sum_{i \in I^c} y_i^{-\theta_0} + \sum_{i \in I} (x_i \wedge y_i)^{-\theta_0}}%
            {\sum_{i \in I} x_i^{-\theta_0}}
    \biggr)^{-1/\theta_0} & \mbox{if $0 < \theta_0 < \infty$,} \\[1ex]
    \displaystyle
    \frac{\bigwedge_{i=1}^d y_i \wedge \bigwedge_{i \in I} x_i}{\bigwedge_{i \in I} x_i}
    & \mbox{if $\theta_0 = \infty$.}
    \end{array}
        \right.
\end{align*}
\end{cor}

When viewed as a function of $(y_1, \ldots, y_d)$, the right-hand side of the previous display is a $d$-variate distribution function. Its copula is the Clayton copula \citep{clayton} with parameter $\theta_0$, see model~(1) in Table~\ref{T}. 

\subsection{Asymptotic independence: non-strict generators}
\label{SS:lower:AI:nonstrict}

If $\phi(0) < \infty$, then necessarily $\theta_0 = 0$ in \eqref{E:lower:RV} and therefore also in \eqref{E:lower:AD:RV}. By definition of $\phi\inv$, we have $\phi\inv(y) = 0$ if $y \in [\phi(0), \infty]$. So if $s \in (0, s_0]$ where $s_0 = \phi\inv(\phi(0)/2)$, then actually
\[
	\Pr[U_i \leq s, U_j \leq s] = \phi\inv(2\phi(s)) \leq \phi\inv(2\phi(s_0)) = \phi\inv(\phi(0)) = 0
\]
for all integer $1 \leq i < j \leq d$. This is obviously much stronger than \eqref{E:lower:AD} with $\theta_0 = 0$.

\subsection{Asymptotic independence: strict generators}
\label{SS:lower:AI:strict}

Suppose that $\phi$ is strict, that is, $\phi(0) = \infty$. If $\theta_0 = 0$ in \eqref{E:lower:AD:RV}, then by Theorem~\ref{T:lower:AD},
\begin{equation}
\label{E:lower:AI:0}
	\Pr[ \forall i \in I : U_i \leq sx_i ] = o(s), \qquad s \downarrow 0,
\end{equation}
whenever $I \subset \{1, \ldots, d\}$ has at least two elements and $0 < x_i < \infty$. In contrast to \eqref{E:lower:AD} with $\theta_0 > 0$, the above display does not give the precise rate of convergence to zero of the probability on the left-hand side. Similarly, it does not permit calculation of the limit distribution of the appropriately normalized vector $(U_1, \ldots, U_d)$ conditionally on the event $\{ \forall i \in I : U_i \leq s x_i\}$ as $s \downarrow 0$ where $\varnothing \neq I \subset \{1, \ldots, d\}$.

The following theorem gives a more precise statement on the rate of convergence in \eqref{E:lower:AI:0}. The result requires an additional assumption on the generator in the neighbourhood of zero, or equivalently on its inverse in the neighbourhood of infinity. The assumption is verified for all models in Table~\ref{T} for which $\phi(0) = \infty$ and $\theta_0 = 0$.

\begin{thm}
\label{T:lower:AI:1}
If $\phi(0) = \infty$, if \eqref{E:lower:AD:RV} holds with $\theta_0 = 0$, and if the function $\psi = -1 / D(\log\phi\inv)$ is regularly varying at infinity of finite index $\kappa$, then $\kappa \leq 1$, and for every $\varnothing \neq I \subset \{1, \ldots, d\}$ and every $(x_i)_{i \in I} \in (0, \infty)^{|I|}$,
\begin{equation}
\label{E:lower:AI:1}
	\lim_{s \downarrow 0} \frac{1}{\phi\inv(|I|\phi(s))} \Pr[ \forall i \in I : U_i \leq s x_i]
	= \prod_{i \in I} x_i^{|I|^{-\kappa}}.
\end{equation}
\end{thm}

By \eqref{E:lower:AI:1}, the probability on the left-hand side of \eqref{E:lower:AI:0} is not only $o(s)$ but of the (precise) order $\phi\inv(|I| \phi(s))$ as $s \downarrow 0$. The latter function is regularly varying at zero with index $|I|^{1-\kappa}$ [take $x_i = x$ in \eqref{E:lower:AI:1}]. Specializing to the case where $|I| = 2$, we obtain the pairwise index of (lower) tail dependence introduced by Ledford and Tawn \citep{Ledford96}: for $i \neq j$,
\[
	\eta_L 
	= \lim_{s \downarrow 0} \frac{\log s}{\log \Pr[U_i \leq s, U_j \leq s]} 
	= \lim_{s \downarrow 0} \frac{\log s}{\log \phi^\leftarrow(2\phi(s))} 
	= 2^{\kappa - 1}.
\]
For Archimedean copulas, the case $\kappa = 0$ ($\eta_L = 1/2$) occurs relatively often, a prime example being the independent copula. If $0 < \kappa \leq 1$ ($\eta_L > 1/2$), then the lower tail of $C$ is heavier than the one of the independent copula, while if $\kappa < 0$ ($\eta_L < 1/2$), then the converse is true.

Fixing a single $i \in I$ and letting $x_i \to \infty$ on both sides of \eqref{E:lower:AI:1} leads to the conclusion that for integer $1 \leq j < k \leq d$, the function $\phi\inv(k \phi(s))$ is of smaller order than $\phi\inv(j \phi(s))$ as $s \downarrow 0$. That is, if $j < k$, then the probability that $k$ variables are small simultaneously is of smaller order than the probability that only $j$ variables are small simultaneously. Therefore, Theorem~\ref{T:lower:AI:1} still does not say anything on the conditional distribution given $U_i \leq s x_i$ for all $i \in I$ of the remaining variables $U_i$ with $i \not\in I$. The following theorem does.

\begin{thm}
\label{T:lower:AI:2}
Under the conditions of Theorem~\ref{T:lower:AI:1}, for every $\varnothing \neq I \subset \{1, \ldots, d\}$, every $(x_i)_{i \in I} \in (0, \infty)^{|I|}$ and every $(y_1, \ldots, y_d) \in (0, \infty)^d$,
\begin{eqnarray}
\label{E:lower:AI:2}
    \lefteqn{
    \lim_{s \downarrow 0}
    \Pr[ \forall i \in I : U_i \leq s y_i ; \forall i \in I^c : U_i \leq \chi_s(y_i) 
    \mid \forall i \in I : U_i \leq s x_i ]
    } \nonumber \\
    &=& \prod_{i \in I} \biggl( \frac{y_j}{x_j} \wedge 1 \biggr)^{|I|^{-\kappa}} 
    \prod_{i \in I^c} \exp\left(-|I|^{-\kappa} y_i^{-1}\right),
\end{eqnarray}
where $\chi_s(y) = \phi\inv(y^{-1} \psi(\phi(s)))$, a function which has the following properties:
\begin{description}
\item[{\it (i)}]
the map $[0, \infty] \to [0, 1] : y \mapsto \chi_s(y)$ is an increasing homeomorphism for all $0 < s < 1$;
\item[{\it (ii)}]
$\lim_{s \downarrow 0} s / \chi_s(y) = 0$ for all $0 < y < \infty$.
\end{description}
\end{thm}

According to Theorem~\ref{T:lower:AI:2}, conditionally on the event that $U_i \leq s x_i$ for all $i \in I$, the proper normalization for the remaining variables $U_i$ with $i \not\in I$ is given by the function $\chi_s(\,\cdot\,)$. Moreover, by \emph{(ii)}, the variables $U_i$ with $i \not\in I$ are of larger order than the variables $U_i$ with $i \in I$. Finally, since the limit in \eqref{E:lower:AI:2} factorizes in the $y_i$, the limiting conditional distribution of the appropriately normalized vector $(U_1, \ldots, U_d)$ given $U_i \leq s x_i$ for all $i \in I$ has independent marginals. This is a rather strong form of asymptotic independence.

\begin{rem} \rm
If the index of regular variation of $\psi$ is $\kappa = - \infty$, then Theorems~\ref{T:lower:AI:1} and~\ref{T:lower:AI:2} do not apply. Still, one can show that the function $s \mapsto C(s, \ldots, s) = \phi\inv(d\phi(s))$ is regularly varying at zero with index $\infty$. In particular, $C(s, \ldots, s) = o(s^p)$ as $s \downarrow 0$ for every exponent $p \in (0, \infty)$. In this sense, the lower tail of $C$ contains very little probability mass.
\end{rem}

\section{Upper tail}
\label{S:upper}

Let $C$ be an Archimedean copula with generator $\phi$. Let $\overline{C}$ be the survival copula of $C$, that is, $\overline{C}(u_1, \ldots, u_d) = \Pr[U_1 > 1-u_1, \ldots, U_d > 1-u_d]$. Results in this section concern the behavior of $\overline{C}(u_1, \ldots, u_d)$ when at least one of the coordinates $u_i$ tends to $0$. This time, what matters is the behaviour of the generator $\phi$ in the neighbourhood of $1$. 

We assume the existence of the limit in $[1, \infty]$ of
\begin{equation}
\label{E:upper:AD:RV}
	\theta_1 := - \lim_{s \downarrow 0} \frac{s \phi'(1-s)}{\phi(1-s)}.
\end{equation}
The existence of the limit is not a very restrictive assumption as it is satisfied by virtually every parametric model. Moreover, by convexity, $\phi(1-s) \leq -s \phi'(1-s)$, so that indeed necessarily $\theta_1 \geq 1$. By the monotone density theorem (Lemma~\ref{L:RV}), equation~\eqref{E:lower:AD:RV} is equivalent to regular variation of the function $s \mapsto \phi(1-s)$ at $0$ with index $\theta_1$: 
\[
	\lim_{s \downarrow 0} \frac{\phi(1-st)}{\phi(1-s)} = t^{\theta_1}, \qquad t \in (0, \infty).
\]

There are two major cases: if $\theta_1 > 1$, then the upper tail exhibits \emph{asymptotic dependence} (Subsection~\ref{SS:upper:AD}), while if $\theta_1 = 1$, the upper tail exhibits \emph{asymptotic independence}. The latter cases branches out further in two subcases, depending on whether
\[
	\lim_{s \downarrow 0} \frac{\phi(1-s)}{s} = -\phi'(1)
\]
is positive or zero. (By convexity, the limit in the above display always exists.) On the one hand, if $\phi'(1) < 0$, then there is asymptotic independence in a rather strong sense, a case which is called \emph{near independence} in \citep{Ledford97} (Subsection~\ref{SS:upper:NI}). On the other hand, if $\phi'(1) = 0$ and $\theta_1 = 1$, we are on the boundary between asymptotic independence and asymptotic dependence, a case which we coin \emph{near asymptotic dependence} (Subsection~\ref{SS:upper:NAD}). In terms of Ledford and Tawn's \cite{Ledford96} index of (upper) tail dependence, we have
\begin{eqnarray*}
	\eta_U
	&=& \lim_{s \downarrow 0} \frac{\log s}{\log \Pr[U_i \geq 1-s, U_j \geq 1-s]} \\
	&=& 
	\left\lbrace
		\begin{array}{l@{\quad}l}
			1/2 & \mbox{if $\phi'(1) < 0$ (near independence; Subsection~\ref{SS:upper:NI}),} \\
			1 & \mbox{if $\phi'(1) = 0$ (near asymptotic dependence; Subsection~\ref{SS:upper:NAD}).}
		\end{array}
	\right.
\end{eqnarray*}

Note that if $\phi'(1) < 0$, then by convexity, $\theta_1 = 1$, while if $\phi'(1) = 0$, then both $\theta_1 = 1$ and $\theta_1 > 1$ are possible. In other words, if $\theta_1 > 1$, then necessarily $\phi'(1) = 0$, while if $\theta_1 = 1$, then both $\phi'(1) = 0$ and $\phi'(1) < 0$ are possible. The boundary case $\phi'(1) = 0$ and $\theta_1 = 1$ occurs only rarely. Therefore, in order to determine the category to which the upper tail of an Archimedean copula with generator $\phi$ belongs, it is usually simpler to compute $\phi'(1)$ first: if $\phi'(1) < 0$, then automatically $\theta_1 = 1$, and only if $\phi'(1) = 0$ is it necessary to actually compute $\theta_1$. This is the order which is used in the decision tree in Figure~\ref{F}.

The proofs of the theorems in this section are gathered in Section~\ref{S:upper:proofs}.

\subsection{Asymptotic dependence}
\label{SS:upper:AD}

\begin{thm}
\label{T:upper:AD}
If the limit $\theta_1$ in \eqref{E:upper:AD:RV} exists in $[1, \infty]$, then for every $I \subset \{1, \ldots, d\}$ with $|I| \geq 2$ and every $(x_i)_{i \in I} \in (0,\infty)^{|I|}$,
\begin{align}
\label{E:upper:AD}
	\lefteqn{
    \lim_{s \downarrow 0} s^{-1} \Pr[ \forall i \in I : U_i \geq 1 - s x_i ]
    } \nonumber \\
    &= 	\left\lbrace
        \begin{array}{l@{\quad}l}
            0 & \mbox{if $\theta_1=1$}, \\
            {\displaystyle \sum_{\varnothing \neq J \subset I}
            (-1)^{|J|-1} \bigl( \ssum_{i \in J} x_i^{\theta_1} \bigr)^{1/\theta_1}}
            & \mbox{if $1 < \theta_1 < \infty$,} \\
            \bigwedge_{i \in I} x_i
            & \mbox{if $\theta_1=\infty$.}
        \end{array}
	    \right.
\end{align}
\end{thm}

By \eqref{E:upper:AD}, the index of upper tail dependence of an arbitrary pair of variables is
\[
	\lambda_{U}
	= \lim_{s \uparrow 1} \Pr[U_i >1- s \mid U_j > 1- s] 
	= 2 - 2^{1/\theta_1},
\]%
for $i \neq j$ and where $\lambda_U$ is to be interpreted as $1$ if $\theta_1$ is $\infty$. First, if $\theta_1 = 1$, then $\lambda_U = 0$, that is, every pair of variables has an asymptotically independent upper tail. The precise behavior of the joint upper tail now depends on whether $\phi'(1) < 0$ (Subsection~\ref{SS:upper:NI}) or $\phi'(1) = 0$ (Subsection~\ref{SS:upper:NAD}).

Second, if $\theta_1 > 1$, then a straightforward computation yields
\[
	\lim_{x_i \to \infty} \lim_{s \downarrow 0} \Pr[U_i < 1 - sx_i \mid U_j > 1-s] = 0,
\]
that is, given $U_j$ is close to $1$, all the other variables will be close to $1$ as well. In that case, it is possible to compute the limit distribution of the vector $(s^{-1}(1-U_1), \ldots, s^{-1}(1-U_d))$ as $s \downarrow 0$ conditionally on the event that $U_i \geq 1 - sx_i$ for all $i$ in some non-empty set $I$.

\begin{cor}
\label{C:upper:AD}
If \eqref{E:upper:AD:RV} holds with $1 < \theta_1 \leq \infty$, then for every $\varnothing \neq I \subset \{1, \ldots, d\}$, every $(x_i)_{i \in I} \in (0, \infty)^{|I|}$ and every $(y_1, \ldots, y_d) \in (0, \infty)^d$,
\begin{equation}
\label{E:upper:AD:UTDC}
	\lim_{s \downarrow 0} 
	\Pr[ \forall i = 1, \ldots, d : U_i \geq 1 - s y_i
	\mid \forall i \in I : U_i \geq 1 - s x_i ]
	= \frac{r_d(z_1, \ldots, z_d; \theta_1)}{r_{|I|}((x_i)_{i \in I}; \theta_1)}
\end{equation}
where $z_i = x_i \wedge y_i$ for $i \in I$ and $z_i = y_i$ for $i \in I^c$ and
\[
    r_k(u_1,\ldots,u_k; \theta_1) =
    \left\lbrace
        \begin{array}{l@{\quad}l}
            {\displaystyle \sum_{\varnothing \neq J \subset \{1,\ldots,k\}}
            (-1)^{|J|-1} \bigl( \ssum_{i \in J} u_j^{\theta_1} \bigr)^{1/\theta_1}}
            & \mbox{if $1 < \theta_1 < \infty$}, \\
            u_1 \wedge \cdots \wedge u_k
            & \mbox{if $\theta_1 = \infty$},
        \end{array}
    \right.
\]
for integer $k \geq 1$ and $(u_1, \ldots, u_k) \in (0, \infty)^k$. 
\end{cor}

When viewed as a function of $(y_1, \ldots, y_d)$, the right-hand side of \eqref{E:upper:AD:UTDC} is a $d$-variate distribution function. Even in simple special cases ($d = 2$, $I = \{1, 2\}$, $x_1 = x_2 = 1$), we have not been able to write down an explicit expression for its copula or its survivor copula, nor to identify one of those two as a member of a known copula family.

\subsection{Asymptotic independence: Near independence}
\label{SS:upper:NI}

If $\theta_1 = 1$ in \eqref{E:upper:AD:RV}, then the conclusion of Theorem~\ref{T:upper:AD} is that
\begin{equation}
\label{E:upper:AI}
    \lim_{s \to 0} s^{-1} \Pr[U_i \geq 1 - sx_i, U_j \geq 1 - sx_j] = 0
\end{equation}
for every $1 \leq i < j \leq d$ and every $x_i, x_j \in (0, \infty)$. This statement is not very informative as the rate of convergence to zero can be arbitrarily slow or fast. The present section and the next one attempt to give more precise results. There are two qualitatively different subcases, depending on whether $\phi'(1) < 0$ (this subsection) or $\phi'(1) = 0$ (Subsection~\ref{SS:upper:NAD}). Recall that $\phi\inv$ is the (generalized) inverse of $\phi$. Since $\phi(1) = 0$, the behavior of $\phi$ near $1$ and the one of $\phi\inv$ near $0$ mutually determine each other.

\begin{thm}
\label{T:upper:AI:I}
Let $\varnothing \neq I \subset \{1, \ldots, d\}$. If $\phi\inv$ is $|I|$ times continuously differentiable and if $(-D)^{|I|} \phi\inv(0) < \infty$, then $\phi'(1) < 0$ and
\begin{align*}
	\lefteqn{
	\lim_{s \downarrow 0} s^{-|I|}
    \Pr[ \forall i \in I : U_i \geq 1 - s x_i ; \forall i \in I^c : U_i \leq y_i ]
    } \\
    &= |\phi'(1)|^{|I|} \prod_{i \in I} x_i
    \cdot (-D)^{|I|} \phi\inv \bigl( \ssum_{i \in I^c} \phi(y_i) \bigr) 
\end{align*}
whenever $0 < x_i < \infty$ for $i \in I$ and $0 < y_i \leq 1$ for $i \in I^c$.
\end{thm}

The particular case $y_i = 1$ for all $i \in I^c$ yields
\[
	\lim_{s \downarrow 0} s^{-|I|} \Pr[ \forall i \in I : U_i \geq 1 - s x_i]
	= |\phi'(1)|^{|I|} (-D)^{|I|} \phi\inv(0) \prod_{i \in I} x_i.
\]
So if also $(-D)^{|I|} \phi\inv(0) > 0$, then the joint survivor function of $(U_i)_{i \in I}$ is proportional to the one of the independence copula. In this sense, the case $\phi'(1) < 0$ corresponds to a particularly strong form of asymptotic independence.

The asymptotic conditional distribution of $(U_1, \ldots, U_d)$ given that $U_i \geq 1 - s x_i$ for all $i \in I$ follows from Theorem~\ref{T:upper:AI:I} at once.

\begin{cor}
\label{C:upper:AI:I}
Under the conditions of Theorem~\ref{T:upper:AI:I}, if also $(-D)^{|I|} \phi\inv(0) > 0$, then for all $(x_i)_{i \in I} \in (0, \infty)^{|I|}$ and $(y_1, \ldots, y_d) \in (0, 1]^d$ ,
\begin{eqnarray*}
	\lefteqn{
	\lim_{s \downarrow 0}
	\Pr[ \forall i \in I : U_i \geq 1 - s y_i ; \forall i \in I^c : U_i \leq y_i \mid
		\forall i \in I : U_i \geq 1 - s x_i ]
	} \\
	&=& \prod_{i \in I} y_j \cdot
		\frac{(-D)^{|I|}\phi\inv (\sum_{i \in I^c} \phi(y_i))}{(-D)^{|I|}\phi\inv(0)}.
\end{eqnarray*}
\end{cor}

If $|I^c| \geq 2$ in Corollary~\ref{C:upper:AI:I}, the copula of the limiting conditional distribution of $(U_i)_{i \in I^c}$ given $U_i \geq 1 - sx_i$ for all $i \in I$ is Archimedean with generator
\[
    \phi_{|I|} = \biggl( \frac{(-D)^{|I|}\phi\inv(\,\cdot\,)}{(-D)^{|I|}\phi\inv(0)} \biggr)\inv.
\]

\subsection{Asymptotic independence: Near asymptotic dependence}
\label{SS:upper:NAD}

In this subsection, we treat the case $\theta_1 = 1$ in \eqref{E:upper:AD:RV} and, simultaneously, $\phi'(1) = 0$. From Theorem~\ref{T:upper:AD} it follows that the upper tail of $C$ is asymptotically independent. Although this case does usually not occur for parametric models used in practice, we still include it in this taxonomy as the results in this case are somewhat surprising and interesting in their own right.

We begin with the description of the asymptotic distribution of the vector $(U_1, \ldots, U_d)$ given that one component is small. Since the law of $(U_1, \ldots, U_d)$ is exchangeable, we can without loss of generality fix this component to be $U_1$. 

\begin{thm}\label{T:upper:AI:com}
If $\phi'(1) = 0$ and \eqref{E:upper:AD:RV} holds with $\theta_1 = 1$, then the function $s \mapsto \ell(s) = s^{-1} \phi(1-s)$ is increasing and slowly varying at zero, and for $(x_1, \ldots, x_d) \in (0, 1]^d$,
\begin{eqnarray}
\label{E:upper:AI:com}
    \lefteqn{
    \lim_{s \downarrow 0} 
    \Pr[ U_1 > 1 - sx_1 ; \forall i = 2, \ldots, d : U_j \leq 1 - \eta_s(x_j) \mid U_1 > 1-s ]
    } \nonumber \\
    &=& x_1 \min(x_2,\ldots,x_d).
\end{eqnarray}
where $\eta_s(x) = \ell\inv(x^{-1}\ell(s))$, a function which has the following properties:
\begin{description}
\item[\it (i)] $\lim_{s \downarrow 0} \eta_s(x) = 0$ for all $0 < x < \infty$;
\item[\it (ii)] $\lim_{s \downarrow 0} s / \eta_s(x) = 0$ for all $0 < x < 1$.
\end{description}
\end{thm}

The conclusion of Theorem~\ref{T:upper:AI:com} implies that, conditionally on $U_1 > 1 - s$ with $s \downarrow 0$, every $U_i$ with $i \neq 1$ converges in probability to one but at a slower rate than $s$, that is, for every $0 < \eps < 1$, $1 < \lambda < \infty$, and $i \in \{2, \ldots, d\}$,
\[
    \lim_{s \downarrow 0} \Pr[ 1 - \eps < U_i < 1 - s \lambda \mid U_1 > 1 - s] = 1.
\]
Moreover, in the limit, the vector $(U_2, \ldots, U_d)$ is asymptotically independent from $U_1$ but is itself comonotone. Note that this is completely different from the case $|I| = 1$ in Corollary~\ref{C:upper:AI:I}.

Next, we study the joint survival function of the vector $(U_1, \ldots, U_d)$. A precise asymptotic result on the probability that all $U_i$ are close to unity simultaneously is possible under a certain refinement of the condition that the function $s \mapsto \phi(1-s)$ is regularly varying at zero of index one. We need the following two auxiliary functions defined for $0 < s < 1$:
\begin{align}
\label{E:upper:AI:L}
	L(s) &:= s \frac{\mathrm{d}}{\mathrm{d}s} \{ s^{-1} \phi(1-s) \} = -\phi'(1-s) - s^{-1}\phi(1-s), \\
\label{E:upper:AI:g}
	g(s) &:= \frac{s \, L(s)}{\phi(1-s)} = - \frac{s \phi'(1-s)}{\phi(1-s)} - 1. 
\end{align}

\begin{thm}
\label{T:upper:AI:eta=1}
If $\phi'(1) = 0$ and if the function $L$ in \eqref{E:upper:AI:L} is positive and slowly varying at zero, then the function $g$ in \eqref{E:upper:AI:g} is positive and slowly varying at zero as well, $g(s) \to 0$ as $s \downarrow 0$, and for $(x_1, \ldots, x_d) \in (0, \infty)^d$,
\begin{align}
\label{E:upper:AI:eta=1:r}
	\lefteqn{
    \lim_{s \to 0} \frac{1}{s g(s)} \Pr[U_1 \geq 1 - sx_1, \ldots, U_d \geq 1 - sx_d]
    } \nonumber \\
    &= \sum_{\varnothing \neq I \subset \{1,\ldots,d\}}
        (-1)^{|I|} (\ssum_I x_i) \log(\ssum_I x_i) \nonumber \\
    &= (d-2)! \int_0^{x_1} \cdots \int_0^{x_d} (t_1 + \cdots + t_d)^{-(d-1)}
        \mathrm{d}t_1 \cdots \mathrm{d}t_d.
\end{align}
\end{thm}

The case $d=2$ of Theorem~\ref{T:upper:AI:eta=1} provides examples of copulas for which the coefficient of upper tail dependence is equal to zero and at the same time Ledford and Tawn's index of tail dependence, $\eta$, is equal to one \citep{Ledford97}. The case of general $d$ in Theorem~\ref{T:upper:AI:eta=1} provides examples of distributions exhibiting hidden regular variation with a non-trivial hidden angular measure \citep{Maulik04, Resnick02}.

A simple sufficient condition for the function $L$ in \eqref{E:upper:AI:L} to satisfy the requirements in Theorem~\ref{T:upper:AI:eta=1} is that the function $s \mapsto \phi(1-s)$ is twice continuously differentiable and that its second derivative is positive and regularly varying at zero of index $-1$. Under the conditions of Theorem~\ref{T:upper:AI:eta=1}, it follows from \eqref{E:upper:AI:aux:20} in the proof that for all $0<x<\infty$,
\[
	\frac{\phi(1-sx)}{\phi(1-s)} = x + g(s) x \log x + o\{g(s)\}, \qquad s \downarrow 0,
\]
that is, the function $s \mapsto \phi(1-s)$ is second-order regularly varying at zero with index one and auxiliary function $g$.

\begin{cor}
\label{C:upper:AI:eta=1}
Under the assumptions of Theorem~\ref{T:upper:AI:eta=1}, if $I \subset \{1, \ldots, d\}$ and $|I| \geq 2$, then for every $(x_i)_{i \in I} \in (0, \infty)^{|I|}$ and every $(y_1, \ldots, y_d) \in (0, \infty)^d$,
\[
    \lim_{s \downarrow 0} 
	\Pr[ \forall i = 1, \ldots, d : U_i \geq 1 - sx_i y_i
	\mid \forall i \in I : U_i \geq 1 - s x_i ]
    = \frac{r_d(z_1,\ldots,z_d)}{r_{|I|}((x_i)_{i \in I})}
\]
where $z_i = x_i \wedge y_i$ for $i \in I$ and $z_i = y_i$ for $i \in I^c$ and
\begin{align*}
    r_k(u_1,\ldots,u_k) 
    &:= \sum_{\varnothing \neq J \subset \{1,\ldots,k\}}
        (-1)^{|J|} (\ssum_J u_j) \log(\ssum_J u_j) \\
    &= (k-2)! \int_0^{u_1} \cdots \int_0^{u_k} (t_1 + \cdots + t_k)^{-(k-1)}
        \mathrm{d}t_1 \cdots \mathrm{d}t_k
\end{align*}
for integer $k \geq 2$ and $(u_1, \ldots, u_k) \in (0, \infty)^k$.
\end{cor}

The limit distribution in Corollary~\ref{C:upper:AI:eta=1} is quite different from the one in Theorem~\ref{T:upper:AI:com}. Even in the simple case $d = 2$, $I = \{1, 2\}$ and $x_1 = x_2 = 1$, we have not been able to identify this distribution or compute its (survival) copula. 

\section{Lower tail: Proofs}
\label{S:lower:proofs}

We present the proofs of the theorems in Section~\ref{S:lower}.

\subsection{Proof of Theorem~\ref{T:lower:AD}}

By Lemma~\ref{L:RV} below, equation~\eqref{E:lower:AD:RV} is equivalent to regular variation of $\phi$ at zero of index $- \theta_0$.

The distribution function of $(U_i)_{i \in I}$ is given by the $|I|$-variate copula with generator $\phi$. Hence, it suffices to show \eqref{E:lower:AD} for the case $I = \{1, \ldots, d\}$.

First, suppose $\theta_0 = 0$. Fix $0 < \eps < 1$. Since $\phi$ is slowly varying, $\phi(s\eps) \leq 2\phi(s)$ and thus $s\eps \geq \phi\inv(2\phi(s))$ for all sufficiently small $s > 0$. Hence $\phi\inv\{2\phi(s)\} = o(s)$ as $s \downarrow 0$. Denoting $x = x_1 \vee \cdots \vee x_d$, we arrive at $C(sx_1,\ldots,sx_d) \leq \phi\inv(d\phi(sx)) \leq \phi\inv(2\phi(sx)) = o(s)$ as $s \downarrow 0$.

Second, suppose $0 < \theta_0 < \infty$. The function $t \mapsto f(t) = \phi(1/t)$ is increasing and regularly varying at infinity of index $\theta_0$. By \citet[Theorem~1.5.12]{BGT}, its inverse, $f\inv$ is regularly varying at infinity of index $1/\theta_0$. Since $\phi\inv = 1/f\inv$, we find that $\phi\inv$ is regularly varying at infinity of index $-1/\theta_0$. Now write
\[
    s^{-1} C(sx_1,\ldots,sx_d)
    =	\frac{1}{\phi\inv(\phi(s))}
        \phi\inv
            \biggl( \phi(s)
                \biggl\{
                    \frac{\phi(sx_1)}{\phi(s)} + \cdots + \frac{\phi(sx_d)}{\phi(s)}
                \biggr\}
            \biggr).
\]
Since $\phi(s) \to \infty$ as $s \downarrow 0$ and by the uniform convergence theorem for regularly varying functions \citep[Theorem~1.5.2]{BGT}, the right-hand side of the previous display converges to $(x_1^{-\theta_0} + \cdots + x_d^{-\theta_0})^{-1/\theta_0}$ as $s \downarrow 0$.

Finally, suppose $\theta_0 = \infty$. Denote $m = x_1 \wedge \cdots \wedge x_d$. We have
\[
    s^{-1} \phi\inv(d\phi(sm)) \leq s^{-1} C(sx_1, \ldots, sx_d) \leq m.
\]
Fix $0 < \lambda < 1$. Since $\phi$ is regularly varying at zero with index $-\infty$, we have $\phi(\lambda sm) \geq d \phi(sm)$ and thus $\lambda m \leq s^{-1} \phi\inv(d\phi(sm))$ for all sufficiently small $s > 0$. Let $\lambda$ increase to one to see that $\lim_{s \downarrow 0} s^{-1} C(sx_1, \ldots, sx_d) = m$. This finishes the proof of Theorem~\ref{T:lower:AD}. \halmos

\subsection{Proofs of Theorems~\ref{T:lower:AI:1} and \ref{T:lower:AI:2}}

Theorems~\ref{T:lower:AI:1} and \ref{T:lower:AI:2} are both special cases of the following one.

\begin{thm}
\label{T:lower:AI}
Under the conditions of Theorem~\ref{T:lower:AI:1}, necessarily $\kappa \leq 1$, and for every $\varnothing \neq I \subset \{1, \ldots, d\}$ and every $(x_1, \dots, x_d) \in (0,\infty)^d$,
\begin{eqnarray}
\label{E:lower:AI}
    \lefteqn{
    \lim_{s \downarrow 0} \frac{1}{\phi\inv(|I|\phi(s))} 
    \Pr[ \forall i \in I : U_i \leq s x_i ; \forall i \in I^c : U_i \leq \chi_s(x_i) ]
    } \\
    &= \prod_{i \in I} x_i^{|I|^{-\kappa}}
    \prod_{i \in I^c} \exp (-|I|^{-\kappa} x_i^{-1}), \nonumber
\end{eqnarray}
with $\chi_s(\,\cdot\,)$ as in Theorem~\ref{T:lower:AI:1}.
\end{thm}

\begin{proof}
By the chain rule, $\psi(t) = - \phi\inv(t) \phi'(\phi\inv(t))$ and thus $\psi(\phi(s)) = - s \phi'(s)$. Equation~\eqref{E:lower:AD:RV} with $\theta_0 = 0$ therefore implies $\lim_{s \downarrow 0} \psi(\phi(s)) / \phi(s) = 0$ and thus $\lim_{t \to \infty} \psi(t) / t = 0$, whence $\kappa \leq 1$. Moreover, since $\psi \circ \phi$ is slowly varying at zero, the function $-\phi'(s) = s^{-1} \psi(\phi(s))$ is regularly varying at zero of index $-1$.

For $x > 0$, since $\psi(\phi(s)) = - s \phi'(s)$,
\[
	\frac{\phi(sx) - \phi(s)}{\psi(\phi(s))}
	= \frac{1}{\psi(\phi(s))} \int_s^{sx} \phi'(u) \du
	= - \int_1^x \frac{\phi'(vs)}{\phi'(s)} \dv.
\]
Since $-\phi'$ is regularly varying at zero of index $-1$, if $x(s) \to x > 0$ as $s \downarrow 0$, then 
\begin{equation}
\label{E:Pi}
	\lim_{s \downarrow 0}
	\frac{\phi(sx(s)) - \phi(s)}{\psi(\phi(s))}
	= - \log(x)
\end{equation}
by the uniform convergence theorem \citep[Theorem~1.5.2]{BGT}. Equation~\eqref{E:Pi} states that $\phi$ is in the de Haan class $\Pi$ with auxiliary function $\psi \circ \phi$ \citep[section~3.7]{BGT}

Since $\psi(t) = o(t)$ as $t \to \infty$ and since $\psi$ is regularly varying at infinity of finite index, by the uniform convergence theorem \citep[Theorem~1.5.2]{BGT},
\begin{equation}
\label{E:selfneglecting}
	\lim_{t \to \infty} \frac{\psi(t + v \psi(t))}{\psi(t)} = 1
\end{equation}
locally uniformly in $v \in \RR$. Equation~\eqref{E:selfneglecting} states that $\psi$ is self-neglecting \citep[section~2.11]{BGT}.

The function $\phi\inv$ can be expressed in terms of $\psi$: as $\log \phi\inv$ is absolutely continuous and $\log \phi\inv(0) = 0$,
\[
	\phi\inv(t) = \exp \biggl( - \int_0^t \frac{\du}{\psi(u)} \biggr),
\]
for $t \geq 0$. Hence, for real $y$, 
\[
	\frac{\phi\inv(t + y \psi(t))}{\phi\inv(t)}
	= \exp \biggl( - \int_t^{t + y \psi(t)} \frac{\du}{\psi(u)} \biggr) 
	= \exp \biggl( - \int_0^y \frac{\psi(t)}{\psi(t + v \psi(t))} \dv \biggr).
\]
Therefore, if $y(t) \to y \in \RR$ as $t \to \infty$, by \eqref{E:selfneglecting},
\begin{equation}
\label{E:Gamma}
	\lim_{t \to \infty} 
	\frac{\phi\inv(t + y(t)\psi(t))}{\phi\inv(t)}
	= \exp(-y).
\end{equation}
Equation~\eqref{E:Gamma} states that the function $1/\phi$ belongs to the class $\Gamma$ with auxiliary function $\psi$ \citep[section~3.10]{BGT}.

Property~\emph{(i)} of $\chi_s$ stated in Theorem~\ref{T:lower:AI:1} follows from the fact that $\phi\inv$ is a decreasing homeomorphism from $[0, \infty]$ to $[0, 1]$ and $\psi(\phi(s)) = - s \phi'(s) > 0$ for $0 < s < 1$. For property~\emph{(ii)}, take $x > 0$ and $\eps > 0$. Since $\lim_{s \downarrow 0} \phi(s/\eps) / \phi(s) = 1$ and $\lim_{t \to \infty} \psi(t) / t = 0$, there exists $0 < s_0 < 1$ such that $\psi(\phi(s)) \leq x \phi(s/\eps)$ for all $s \in (0, s_0]$. For such $s$, also $\eps \chi_s(x) \geq s$. Since $\eps$ was arbitrary, property~\emph{(ii)} follows. 

Finally, for $I$ and $\xb = (x_1, \ldots, x_d)$ as in the statement of the theorem,
\begin{align}
\label{E:lower:AI:10}
	\lefteqn{
    \Pr[ \forall i \in I : U_i \leq s x_i ; \forall i \in I^c : U_i \leq \chi_s(x_i) ]
    } \nonumber \\
    &=  \phi\inv
        \biggl(
            \sum_{i \in I} \phi(sx_i) + \sum_{i \in I^c} x_i^{-1} \psi(\phi(s))
        \biggr) \nonumber \\
    &= \phi\inv \bigl( |I| \phi(s) + y(\xb; s) \psi(|I|\phi(s)) \bigr)
\end{align}
with
\[
    y(\xb;s) =
    \biggl(
        \sum_{i \in I} \frac{\phi(sx_i) - \phi(s)}{\psi(\phi(s))}
        + \sum_{i \in I^c} x_i^{-1}
    \biggr)
    \frac{\psi(\phi(s))}{\psi(|I|\phi(s))}.
\]
By~\eqref{E:Pi}, since $\phi(0) = \infty$ and since $\psi$ is regularly varying at infinity of index $\kappa$,
\begin{equation}
\label{E:lower:AI:20}
    \lim_{s \downarrow 0} y(\xb;s) =
    \biggl( - \sum_{i \in I} \log(x_i) + \sum_{i \in I^c} x_i^{-1} \biggr) |I|^{-\kappa}.
\end{equation}
Combine equations \eqref{E:Gamma}, \eqref{E:lower:AI:10} and \eqref{E:lower:AI:20} to arrive at~\eqref{E:lower:AI}.
\end{proof}

\section{Upper tail: Proofs}
\label{S:upper:proofs}

\subsection{Proof of Theorem~\ref{T:upper:AD}}

By the (inverse) inclusion-exclusion formula, i.e.\
\[
	\Pr[ \sbigcap_{i \in I} A_i ] 
	= \sum_{\varnothing \neq J \subset I} (-1)^{|J| - 1} \Pr[ \sbigcup_{j \in J} A_j ]
\]
valid for finite $I$ and for arbitrary events $A_i$, eq.\ \eqref{E:upper:AD} follows from the following one, where the intersection ($\forall$) of the events $U_i \geq 1 - s x_i$ has been replaced by a union ($\exists$):
\begin{equation}
\label{E:upper:AD:l}
    \lim_{s \to 0} s^{-1} \Pr[ \exists i \in I : U_i \geq 1 - s x_i ]
    =	\left\lbrace
        \begin{array}{l@{\quad}l}
            (\sum_{i \in I} x_i^{\theta_1})^{1/\theta_1}
            & \mbox{if $1 \leq \theta_1 < \infty$,} \\
            \bigvee_{i \in I} x_i
            & \mbox{if $\theta_1=\infty$,}
        \end{array}
	    \right.
\end{equation}
For the case $\theta_1 = \infty$, note that indeed $\bigwedge_{i \in I} x_i = \sum_{\varnothing \neq J \subset I} (-1)^{|J| - 1} \Pr[ \bigvee_{j \in J} x_j]$. Further, since the copula of the vector $(U_i)_{i \in I}$ is the $|I|$-variate Archimedean copula with generator $\phi$, we can without loss of generality assume that $I = \{1, \ldots, d\}$.

By Lemma~\ref{L:RV}, eq.\ \eqref{E:upper:AD:RV} is equivalent to regular variation at zero of the function $s \mapsto \phi(1-s)$ with index $\theta_1$.

First, consider the case $1 \leq \theta_1 < \infty$. We have
\begin{align*}
    \lefteqn{
    s^{-1} \Pr( \sbigcup_{i=1}^d \{ U_i \geq 1 - s x_i \} )
    } \\
    &= s^{-1} \{ 1 - \phi\inv (\phi(1 - s x_1) + \cdots + \phi(1 - s x_d)) \} \\
    &= \frac{1}{1 - \phi\inv(\phi(1-s))} \\
    & \qquad \mbox{} \times
        \biggl[ 1 - \phi\inv
            \biggl( \phi(1-s)
            \biggl\{ 
            	\frac{\phi(1-sx_1)}{\phi(1-s)} + \cdots
                + \frac{\phi(1-sx_d)}{\phi(1-s)}
            \biggr\}
            \biggr)
        \biggr].
\end{align*}
The function $x \mapsto 1/\phi(1-1/x)$ is regularly varying at infinity with index $\theta_1$. Therefore, its inverse function, the function $t \mapsto 1/\{1-\phi\inv(1/t)\}$ is regularly varying at infinity with index $1/\theta_1$ \citep[Theorem~1.5.12]{BGT}, and thus the function $1-\phi\inv$ is regularly varying at zero with index $1/\theta_1$. For each $i$, we have $\phi(1-sx_i)/\phi(1-s) \to x_i^{\theta_1}$ as $s \downarrow 0$. By the uniform convergence theorem \citep[Theorem~1.5.2]{BGT}, the right-hand side of the previous display then converges to $(\sum_{i=1}^d x_i^{\theta_1})^{1/\theta_1}$, as required.

Second, consider the case $\theta_1 = \infty$. Pick $1 < \lambda < \infty$. Since $s \mapsto \phi(1 - s)$ is regularly varying at zero of index $\infty$, we have $\lim_{s \downarrow 0} \phi(1 - \lambda s) / \phi(1 - s) = \infty$ and thus
\begin{align*}
    \phi(1 - s(x_1 \vee \cdots \vee x_d))
    &\leq \phi(1-sx_1) + \cdots + \phi(1-sx_d) \\
    &\leq d \phi(1 - s(x_1 \vee \cdots \vee x_d)) \\
    &\leq \phi(1 - \lambda s (x_1 \vee \cdots \vee x_d))
\end{align*}
for all $s$ in a right-neighbourhood of zero. Apply the function $1-\phi\inv$ to the various parts of this inequality, multiply by $s^{-1}$ and let $s$ decrease to zero to find
\begin{align*}
    x_1 \vee \cdots \vee x_d
    &\leq \liminf_{s \downarrow 0} 
	    s^{-1} \Pr( \sbigcup_{i=1}^d \{ U_i \geq 1 - sx_i \} ) \\
    &\leq \limsup_{s \downarrow 0} 
    	s^{-1} \Pr( \sbigcup_{i=1}^d \{ U_i \geq 1 - sx_i \} ) 
    \leq \lambda (x_1 \vee \cdots \vee x_d).
\end{align*}
Let $\lambda$ decrease to finish the proof in case $\theta_1 = \infty$. \halmos

\subsection{Proof of Theorem~\ref{T:upper:AI:I}}

If $(-D)^k \phi\inv(0) < \infty$ for some integer $k \geq 1$, then $(-D)^i \phi\inv(0) < \infty$ for all $i = 1, \ldots, k$. Since $D \phi\inv(0) = 1 / \phi'(1)$, necessarily $\phi'(1) < 0$.

By Lemma~\ref{L:volume},
\begin{align*}
	\lefteqn{
	\Pr[ \sbigcap_{i \in I} \{ U_i \geq 1 - sx_i \} 
    \cap \sbigcap_{i \in I^c} \{ U_i \leq y_i \} ]
    } \\
    &= \int_{\prod_I [0, \phi(1-sx_i)]} (-D)^{|I|}
    \phi\inv \bigl( \ssum_{i \in I^c} \phi(y_i) + \ssum_{i \in i} t_i \bigr)
    \mathrm{d}(t_i)_{i \in I}.
\end{align*}
Change variables $t_i = s u_i$ to find
\begin{align*}
	\lefteqn{
	s^{-|I|} \Pr[ \sbigcap_{i \in I} \{ U_i \geq 1 - sx_i \} 
    \cap \sbigcap_{i \in I^c} \{ U_i \leq x_i \} ]
    } \\
    &= \int_{\prod_I [0, s^{-1} \phi(1 - sx_i)]} (-D)^{|I|}
    \phi\inv \bigl( \ssum_{i \in I^c} \phi(y_i) + s \ssum_{i \in I} u_i \bigr)
    \mathrm{d}(u_i)_{i \in I}.
\end{align*}
Since $(-D)^{|I|} \phi\inv$ is continuous and $s^{-1} \phi(1 - sx_i) \to x_i |\phi'(1)|$ as $s \downarrow 0$, the stated limit now follows from the dominated convergence theorem.

\subsection{Proof of Theorem~\ref{T:upper:AI:com}}

Because $\phi$ is positive, convex, and $\phi(s) = o(s)$ as $s \downarrow 0$, the function $\ell$ is positive, increasing, and vanishes at zero. Moreover, since $s \mapsto \phi(1-s)$ is regularly varying at zero of index one, $\ell$ is slowly varying at zero.

Write $\eta_s(x) = \ell\inv(x^{-1}\ell(s))$. Property (i) is clear from the fact that both $\ell$ and $\ell\inv$ vanish at zero. If $0 < x < 1$, then as $\ell$ is slowly varying, $\ell(s) \geq x \ell(s/\eps)$ and thus $\eps \eta_s(x) \geq s$ for all $\eps > 0$ and all $s$ sufficiently close to zero; property (ii) follows.

By Lemma~\ref{L:volume},
\begin{align}
\label{E:upper:AI:com:5}
    \lefteqn{
    \Pr[\{ U_1 > 1 - sx_1 \} \cap \sbigcap_{j=2}^d \{ U_j \leq 1 - \eta_s(x_j) \} ]
    } \nonumber \\
    &= \int_0^{\phi(1 - sx_1)} 
    (-D)\phi\inv \bigl( \ssum_{j=2}^d \phi(1 - \eta_s(x_j)) + y \bigr) 
    \mathrm{d}y.
\end{align}
Let $m = x_2 \wedge \cdots \wedge x_d$. Since $\phi$ is decreasing and by (ii), for all sufficiently small $s$,
\[
	\phi(1 - \eta_s(m)) 
	\leq \sum_{j=2}^d \phi(1 - \eta_s(x_j)) + \phi(1 - sx_1)
	\leq d \phi(1 - \eta_s(m)).
\]
As $(-D)\phi\inv$ is nonincreasing (for $\phi\inv$ is convex), by \eqref{E:upper:AI:com:5},
\begin{align}
\label{E:upper:AI:com:10}
	\lefteqn{
	\phi(1 - sx_1) \cdot (-D)\phi\inv (d \phi (1 - \eta_s(m)))
	} \nonumber \\
	&\leq \Pr[\{ U_1 > 1 - sx_1 \} \cap \sbigcap_{j=2}^d \{ U_j \leq 1 - \eta_s(x_j) \}] 
	\nonumber \\
	&\leq \phi(1 - sx_1) \cdot (-D)\phi\inv (\phi (1 - \eta_s(m))).
\end{align}
Since $s \mapsto \phi(1-s)$ is regularly varying at zero of index $1$, its inverse, $1 - \phi\inv$, must be regularly varying at zero of index $1$ as well. Moreover, by \eqref{E:upper:AD:RV} with $\theta_1 = 1$,
\begin{equation}
\label{E:upper:AI:com:20}
	(-D) \phi\inv (t) = - \frac{1}{\phi'(\phi\inv(t))}
	\sim \frac{1 - \phi\inv(t)}{t}, \qquad t \downarrow 0.
\end{equation}
As a consequence, the function $(-D) \phi\inv$ is slowly varying at zero. The upper and lower bounds in \eqref{E:upper:AI:com:10} are therefore asymptotically equivalent to each other, whence
\begin{align*}
	\lefteqn{
	\Pr[\{ U_1 > 1 - sx_1 \} \cap \sbigcap_{j=2}^d \{ U_j \leq 1 - \eta_s(x_j) \}
	\mid U_1 > 1-s ]
	} \\
	&= s^{-1} 
	\Pr[\{ U_1 > 1 - sx_1 \} \cap \sbigcap_{j=2}^d \{ U_j \leq 1 - \eta_s(x_j) \} ] \\
	&\sim s^{-1} \phi(1 - sx_1) \cdot (-D) \phi\inv (\phi (1 - \eta_s(m))),
	\qquad s \downarrow 0.
\end{align*}
By \eqref{E:upper:AI:com:20}, the last expression is asymptotically equivalent to
\[
	s^{-1} \phi(1 - sx_1) \cdot \frac{\eta_s(m)}{\phi(1 - \eta_s(m))}
	= x_1 \ell(sx_1) \frac{1}{\ell(\eta_s(m))}
	= x_1 m \frac{\ell(sx_1)}{\ell(s)}.
\]
Since $\ell$ is slowly varying at zero, the proof is complete. \halmos

\subsection{Proof of Theorem~\ref{T:upper:AI:eta=1}}

Denote $f(s) = \phi(1-s)$ for $0 \leq s < 1$. Observe that
\[
	L(s) = f'(s) - \frac{f(s)}{s} = s \frac{\mathrm{d}}{\mathrm{d}s} \frac{f(s)}{s}.
\]
Since $s^{-1}f(s) \to 0$ as $s \downarrow 0$,
\[
    f(s) = s \int_0^s L(t) \frac{\mathrm{d}t}{t}, \qquad 0 \leq s < 1.
\]
Note that the function $g$ can be written as
\begin{align}
\label{E:upper:AI:aux:10}
    g(s) 
    &= \frac{sf'(s)}{f(s)} - 1 = \frac{s L(s)}{f(s)} \\
    &= L(s) \bigg/ \int_0^s L(t) \frac{\mathrm{d}t}{t}
    = 1 \bigg/ \int_0^1 \frac{L(st)}{L(s)} \frac{\mathrm{d}t}{t}. \nonumber
\end{align}
By Fatou's lemma, since $L$ is slowly varying at zero, $g(s) \to 0$ as $s \downarrow 0$. Hence, equation~\eqref{E:upper:AD:RV} holds with $\theta_1 = 1$. As a consequence, $f$ is regularly varying at zero of index one, which in turn by \eqref{E:upper:AI:aux:10} implies that $g$ is slowly varying at zero. Moreover, for every $0 < x < \infty$ and every sufficiently small, positive $s$,
\begin{align}
    f(sx)
    &= sx \int_0^{sx} L(t) \frac{\mathrm{d}t}{t} \nonumber \\
    &= xf(s) + sx \int_s^{sx} L(t) \frac{\mathrm{d}t}{t} \nonumber \\
    &= xf(s) + s L(s) \cdot x \int_1^x \frac{L(st)}{L(s)} \frac{\mathrm{d}t}{t} \nonumber \\
    &= f(s) \biggl( x + g(s) \cdot x \int_1^x \frac{L(st)}{L(s)} \frac{\mathrm{d}t}{t} \biggr). 
\label{E:upper:AI:aux:20}
\end{align}

Fix $\xb \in (0, \infty)^d$. For sufficiently small, positive $s$, define $y(\xb,s)$ by
\[
    f(sx_1) + \cdots + f(sx_d)
    = f\{s(x_1 + \cdots + x_d) + sg(s) y(\xb,s)\}.
\]
Since $f(0) = 0$ and $f$ is increasing and convex, $y(\xb,s)$ is well defined and nonpositive. By \eqref{E:upper:AI:aux:20}, on the one hand
\[
    f(sx_1) + \cdots + f(sx_d)
    = f(s) \biggl( \sum_{i=1}^d x_i
    + g(s) \cdot \sum_{i=1}^d x_i \int_1^{x_i} \frac{L(st)}{L(s)}
                                                \frac{\mathrm{d}t}{t} \biggr),
\]
and on the other hand
\begin{align*}
    \lefteqn{
    f\{s(x_1 + \cdots + x_d) + sg(s) y(\xb,s)\}
    } \\
    &= f\{s a(\xb,s)\}
    = f(s) \biggl( a(\xb,s) + g(s) a(\xb,s) \int_1^{a(\xb,s)} \frac{L(st)}{L(s)}
                                                \frac{\mathrm{d}t}{t} \biggr)
\end{align*}
where
\[
    a(\xb,s) = \sum_{i=1}^d x_i + g(s) y(\xb,s).
\]
From the last four displayed equations it follows that
\[
    \sum_{i=1}^d x_i \int_1^{x_i} \frac{L(st)}{L(s)} \frac{\mathrm{d}t}{t}
    = y(\xb,s) + a(\xb,s) \int_1^{a(\xb,s)} \frac{L(st)}{L(s)}
                                                \frac{\mathrm{d}t}{t}.
\]
The left-hand side of this equation converges to $\sum_1^d x_i \log(x_i)$ by the uniform convergence theorem \citep[Theorem~1.2.1]{BGT}. Since $0 < a(\xb,s) \leq \sum_1^d x_i$, the second term on the right-hand side of the previous equation remains bounded from above as $s \downarrow 0$. Therefore, $y(\xb,s)$ must remain bounded from below as $s \downarrow 0$. Since also $y(\xb,s) \leq 0$, necessarily $y(\xb,s) = O(1)$ as $s \downarrow 0$. But since $g(s) \to 0$ as $s \downarrow 0$, 
\[
    \lim_{s \downarrow 0} a(\xb,s) = \sum_{i=1}^d x_i.
\]
Denote $k(x) = x \log(x)$. Combine the two previous displays to conclude that
\[
    y(\xb) := \lim_{s \downarrow 0} y(\xb,s)
    = \sum_{i=1}^d k(x_i) - k\bigl( \ssum_{i=1}^d x_i \bigr).
\]

Next, by definition of $f$ and $y(\xb, s)$,
\begin{align*}
    \Pr[\sbigcup_{i=1}^d \{ U_i \geq 1 - sx_i \}]
    &= 1 - \phi\inv\{\phi(1-sx_1) + \cdots + \psi(1-sx_d)\} \\
    &= f\inv\{f(sx_1) + \cdots + f(sx_d)\} \\
    &= s(x_1 + \cdots + x_d) + s g(s) y(\xb,s) \\
    &= s(x_1 + \cdots + x_d) + s g(s) y(\xb) + o\{s g(s)\},
    \qquad s \downarrow 0.
\end{align*}
Combine this formula with the inverse inclusion-exclusion formula to get
\begin{align*}
    \lefteqn{
    \Pr[\sbigcap_{i=1}^d \{ U_i \geq 1 - sx_i \}]
    } \\
    &= \sum_{\varnothing \neq I \subset \{1,\ldots,d\}}
        (-1)^{|I|-1} \Pr[\sbigcup_{i \in I} \{ U_i \geq 1 - sx_i \}] \\
    &= \sum_{\varnothing \neq I \subset \{1,\ldots,d\}} (-1)^{|I|-1}
        \bigl\{
            s \ssum_I x_i + sg(s) \ssum_I k(x_i) - sg(s) k \bigl( \ssum_I x_i \bigr)
        \bigr\} + o\{sg(s)\}
\end{align*}
as $s \downarrow 0$. Now for every vector $\yb \in \RR^d$,
\[
    \sum_{\varnothing \neq I \subset \{1,\ldots,d\}} (-1)^{|I|-1}
    \ssum_{i \in I} y_i
    = \sum_{i=1}^d 
    \biggl( \sum_{i \in I \subset \{1, \ldots, d\}} (-1)^{|I|-1} \biggr) y_i
    = 0.
\]
Combine the final two displays to arrive at
\begin{align*}
    \lefteqn{
    \Pr[\sbigcap_{i=1}^d \{ U_i \geq 1 - sx_i \}]
    } \\
    &= sg(s) \sum_{\varnothing \neq I \subset \{1,\ldots,d\}} (-1)^{|I|}
    k(\ssum_I x_i) + o\{sg(s)\},
    \qquad s \downarrow 0.
\end{align*}
This yields the first expression for $r(\xb)$. The second expression for $r(\xb)$ follows from Lemma~\ref{L:calculus} applied to the function $k$; note that $(-D)k(x) = -\log(x)-1$ and $(-D)^d k(x) = (d-2)!x^{-(d-1)}$ for all integer $d \geq 2$. \halmos

\small
\appendix

\section{Regular variation of convex functions}
\label{S:RV}

For $0 < x < \infty$, define $x^\infty$ by $\infty$, $1$, or $0$ according to whether $x$ is larger than, equal to, or smaller than $1$, respectively; similarly, define $x^{-\infty}$ by $0$, $1$, or $\infty$ according to whether $x$ is larger than, equal to, or smaller than $1$, respectively.

A positive, measurable function $f$ defined in a right neighbourhood of zero is said to be {\em regularly varying at zero (from the right) of index $\tau \in [-\infty,\infty]$} if $f(tx)/f(t) \to x^\tau$ as $t \downarrow 0$ for all $0<x<\infty$. In case $\tau$ is equal to zero, then $f$ is said to be {\em slowly varying} at zero. Similarly, a positive, measurable function $f$ defined in some neighbourhood of infinity is called {\em regularly varying at infinity of index $\tau \in [-\infty,\infty]$} if $f(tx)/f(t) \to x^\tau$ as $t\to\infty$ for every positive $x$. In case $\tau$ is equal to zero, then $f$ is said to be {\em slowly varying} at infinity. Clearly, a function $f$ is regularly varying at zero of index $\tau$ if and only if the function $t \mapsto f(1/t)$ is regularly varying at infinity of index $-\tau$.
 
The definition of regular variation involves in principle an infinite set of limit relations. However, if a function is known to be convex, then regular variation of the function is equivalent to a single limit relation. Results of this type are known under the name ``Monotone Density Theorem'' \citep[section~1.7.3]{BGT}. We will need the following instance.

\begin{lem}
\label{L:RV}
Let $f$ be a positive, convex function of a real variable defined in a right-neighbourhood of zero. Let $f'$ be a nondecreasing version of the Radon-Nikodym derivative of $f$. The function $f$ is regularly varying at zero of index $\tau \in [-\infty,\infty]$ if and only if
\[
    \lim_{s \downarrow 0} \frac{s f'(s)}{f(s)} = \tau.
\]
\end{lem}

\begin{proof}
Let $c$ be a positive number such that the domain of $f$ includes the interval $(0,c]$. The function $\log f$ is absolutely continuous with Radon-Nikodym derivative $f'/f$. Denote $\tau(s) = s f'(s) / f(s)$. For $0 < s \leq c$, we have
\[
    f(s) = f(c) \exp \biggl( - \int_s^c \tau(t) \frac{\mathrm{d}t}{t} \biggr).
\]
If additionally $0<x<\infty$ with $x \neq 1$ and if $s$ is such that also $sx \leq c$, then
\[
    \frac{f(sx)}{f(s)}
    = \exp \biggl( \int_s^{sx} \tau(t) \frac{\mathrm{d}t}{t} \biggr)
    = \exp \biggl( \int_1^x \tau(st) \frac{\mathrm{d}t}{t} \biggr).
\]
The argument of the exponent converges to $\tau \log(x)$ as $s \downarrow 0$. Hence $f(sx) / f(s) \to x^\tau$ as $s \downarrow 0$, as required.

Conversely, suppose that $f$ is regularly varying at zero of index $\tau$. By convexity, for all $0 < x < \infty$ and all sufficiently small $s$,
\[
    f(sx) - f(s) \geq s (x - 1) f'(s).
\]
Divide both sides of this inequality by $(x-1)$ and let $s$ decrease to zero to get 
\[
	\begin{array}{rcl@{\quad}l}
    \dsty \limsup_{s \downarrow 0} \frac{s f'(s)}{f(s)}
    &\leq& \dsty \frac{x^\tau-1}{x-1}, & \mbox{for all $1 < x < \infty$;} \\
    \dsty \liminf_{s \downarrow 0} \frac{s f'(s)}{f(s)}
    &\geq& \dsty \frac{x^\tau-1}{x-1}, & \mbox{for all $0 < x < 1$}.
    \end{array}
\]
Since $\lim_{x \to 1} (x^\tau-1)/(x - 1) = \tau$ for all $\tau \in [-\infty,\infty]$, indeed $s f'(s) / f(s) \to \tau$ as $s \downarrow 0$.
\end{proof}

\section{Some useful formulas}
\label{S:form}

\begin{lem}
\label{L:calculus}
Let $k$ be positive integer, $\varnothing \neq I \subset \RR$ be an open interval, and $f : I \to \RR$ be a $(k-1)$ times continuously differentiable function. If $D^{k-1} f$ is absolutely continuous with Radon-Nikodym derivative $D^k f$, then for every $x \in I$ and $(x_1,\ldots,x_k) \in [0,\infty)^k$ for which $x + x_1 + \cdots + x_k \in I$,
\[
    \sum_{K \subset \{1,\ldots,k\}} (-1)^{|K|} f(x + \ssum_{i \in K} x_i)
    = \int_0^{x_1} \cdots \int_0^{x_k} (-D)^k f (x + t_1 + \cdots + t_k)
    \dt_k \cdots \dt_1.
\]
\end{lem}

\begin{proof}
The proof is by induction on $k$. If $k = 1$, the assumption is simply that $f$ is
absolutely continuous with Radon-Nikodym derivative $f'$, and the
formula reduces to
\[
    f(x) - f(x + x_1) = - \int_0^{x_1} f'(x + t_1) \dt_1,
\]
which is just the definition of absolute continuity. Let $k \geq 2$. Distinguish between the cases $k \in K$ and $k \not\in K$ to obtain
\begin{align*}
    \lefteqn{
    \sum_{K \subset \{1,\ldots,k\}} (-1)^{|K|} f\left(x + \ssum_{i \in K} x_i\right)
    } \\
    &= \sum_{K \subset \{1, \ldots, k-1\}} (-1)^{|K|}
        \left\{
            f(x + \ssum_{i \in K} x_i)
			- f(x + \ssum_{i \in K} x_i + x_k)
        \right\}.
\end{align*}
Fix $x_k$ and apply the induction hypothesis to the function $y
\mapsto g(y) = f(y) - f(y + x_k)$ to arrive at
\begin{align*}
    \lefteqn{
    \sum_{K \subset \{1, \ldots, k\}} (-1)^{|K|} f(x + \ssum_{i \in K} x_i)
    } \\
    &= \int_0^{x_1} \cdots \int_0^{x_{k-1}}
            (-D)^{k-1} g (x + t_1 + \cdots + t_{k-1})
        \dt_1 \cdots \dt_{k-1}.
\end{align*}
Since $D^{k-1} f$ is absolutely continuous with Radon-Nikodym derivative $D^k f$, the integrand in the previous display is equal to
\begin{align*}
    \lefteqn{
    (-D)^{k-1} g(x + t_1 + \cdots + t_{k-1})
    } \\
    &= (-D)^{k-1} f (x + t_1 + \cdots + t_{k-1})
        - (-D)^{k-1} f (x + t_1 + \cdots + t_{k-1} + x_k) \\
    &= \int_0^{x_k} (-D)^k f (x + t_1 + \cdots + t_k) \dt_k.
\end{align*}
Combine the two previous displays to arrive at the stated formula.
\end{proof}

\begin{lem}
\label{L:volume}
Let $\ub, \vb \in [0, 1]^d$ be such that $\boldsymbol{0} \neq \ub < \vb$ and write $J = \{ j : u_j > 0 \} \neq \varnothing$. If $\phi\inv$ is $|J|-1$ times continuously differentiable and if $D^{|J|-1} \phi\inv$ is absolutely continuous with Radon-Nikodym derivative $D^{|J|} \phi\inv$, then
\begin{align}
\label{E:volume}
	\lefteqn{
	\Pr[\sbigcap_{j = 1}^d \{u_j < U_j \leq v_j\}]
	} \\
	&= \int_{\prod_{j \in J} [\phi(v_j), \phi(u_j)]}
    (-D)^{|J|} \phi\inv
        \bigl( \ssum_{j \in J^c} \phi(v_j) + \ssum_{j \in J} y_j \bigr)
    \mathrm{d} (y_j)_{j \in J}. \nonumber
\end{align}
\end{lem}

\begin{proof}
By the inclusion-exclusion formula,
\begin{align*}
	\lefteqn{
	\Pr[\sbigcap_{j = 1}^d \{u_j < U_j \leq v_j\}]
	} \\
	&= \sum_{I \subset \{1, \ldots, d\}} (-1)^{|I|} 
	\Pr[ \sbigcap_{i \in I} \{ U_i \leq u_i \} \cap \sbigcap_{i \in I^c} \{ U_i \leq v_i \}].
\end{align*}
Since $u_i = 0$ if $i \not\in J$, the summation can be restricted to $I \subset J$, whence
\begin{align*}
	\Pr[\sbigcap_{j = 1}^d \{u_j < U_j \leq v_j\}]
	&= \sum_{I \subset J} (-1)^{|I|} 
	\Pr[ \sbigcap_{i \in I} \{ U_i \leq u_i \} \cap \sbigcap_{i \in I^c} \{ U_i \leq v_i \}] \\	
	&= \sum_{I \subset J} (-1)^{|I|} 
	\phi\inv \bigl( \ssum_{i \in I} \phi(u_i) + \ssum_{i \in I^c} \phi(v_i) \bigr).
\end{align*}
Denote $\Delta_j = \phi(u_j) - \phi(v_j)$ for $j \in J$; note that $0 < \Delta_j < \infty$. We have
\[
	\Pr[\sbigcap_{j = 1}^d \{u_j < U_j \leq v_j\}]
    = \sum_{I \subset J} (-1)^{|I|}
    \phi\inv \bigl( \ssum_{i=1}^d \phi(v_i) + \ssum_{i \in I} \Delta_i \bigr).
\]
Apply Lemma~\ref{L:calculus} to see that
\begin{align*}
	\Pr[\sbigcap_{j = 1}^d \{u_j < U_j \leq v_j\}]
	&= \int_{\prod_J [0, \Delta_j]} (-D)^{|J|} \phi\inv \bigl( \ssum_{i=1}^d \phi(v_i) + \ssum_{j \in J} t_j \bigr) \mathrm{d}(t_j)_{j \in J}.
\end{align*}
Finally, change variables $y_j = \phi(v_j) + t_j$ to arrive at \eqref{E:volume}.
\end{proof}

\section*{Acknowledgments}

The authors gratefully acknowledge inspiring discussions with Christian Genest as well as constructive comments by three referees and an associate editor.


\end{document}